\documentclass[12pt,a4paper,reqno,final]{amsart} 
\usepackage{amsmath, amssymb, amsfonts, amsthm, verbatim, dsfont, graphicx, enumerate, pifont,tikz} 
\usepackage{hyperref}
\usepackage[notref,notcite]{showkeys} 
\usepackage[all,line,dvips]{xy}

\theoremstyle{theorem}
\newtheorem{thm}{Theorem}
\newtheorem{prop}[thm]{Proposition}
\newtheorem{lemma}[thm]{Lemma}
\newtheorem{cor}[thm]{Corollary}

\theoremstyle{definition}

\theoremstyle{plain}

\newcommand{\R}{\mathbb R} 

\DeclareMathOperator{\id}{id}

\DeclareMathOperator{\Ham}{Ham}

\title{A distance expanding flow on exact Lagrangian cobordism classes}
\author{Mads R. Bisgaard}
\begin{document}
\maketitle

Consider an exact symplectic manifold $(M^{2n},d\lambda)$ which is tame at infinity. We assume throughout that the Liouville flow $\{\psi_t\}_{t\in \R}$ on $(M,d\lambda)$ is complete. Given a fixed closed exact Lagrangian submanifold $L\subset (M,d\lambda)$ we denote by $\mathcal{L}(L)$ the space of all closed exact Lagrangian submanifolds of $(M,d\lambda)$, which are exact Lagrangian cobordant to $L$ (see definition below). In \cite{CorneaShelukhin15} Cornea and Shelukhin proved the existence of a remarkable "cobordism metric" $d_c$ on $\mathcal{L}(L)$. In this note we prove
\begin{thm}
	\label{corexpand1}
	Let $L\subset (M,d\lambda)$ be a closed exact Lagrangian submanifold. Then $\{\psi_t\}_{t\in \R}$ induces a flow $\{\Psi_t\}_{t\in \R}$ on $(\mathcal{L}(L),d_c)$ via $\Psi_t(L'):=\psi_t(L')$. More precisely, $\Psi$ defines a continuous $\R$-action on $\mathcal{L}(L)$ and every $\Psi_t:\mathcal{L}(L)\to \mathcal{L}(L)$ satisfies
	\begin{equation}
	\label{eq1}
	d_c(\Psi_t(L'),\Psi_t(L''))=e^td_c(L',L'') \quad \forall \ L',L''\in \mathcal{L}(L).
	\end{equation}
	In particular the metric space $(\mathcal{L}(L),d_c)$ has infinite diameter.
\end{thm}
Before proving Theorem \ref{corexpand1} we note the following special case.
\begin{cor}
	\label{corexpand2}
	Suppose $L\subset (M,d\lambda)$ in Theorem \ref{corexpand1} satisfies $\lambda|_{L}\equiv 0$. Then $\Psi_t(L)=L$ for all $t\in \R$ and $L$ is the unique fixed point of every $\Psi_{t\neq 0}$. In particular $\Psi$ restricts to a free $\R$-action on $\mathcal{L}(L)\backslash \{L\}$. Hence, if $L'\subset (M,d\lambda)$ is another closed Lagrangian satisfying $\lambda|_{L'}\equiv 0$, then $L'\notin \mathcal{L}(L)$.
\end{cor}
\begin{proof}
	The condition $\lambda|_{L}\equiv 0$ is equivalent to the condition $\psi_t(L)=L$ for all $t\in \R$. If $L'\in \mathcal{L}(L)$ is a fixed point for some $\Psi_{t\neq 0}$ then 
	\[
	d_c(L,L')=d_c(\Psi_t(L),\Psi_t(L'))=e^td_c(L,L').
	\]
	Since $t\neq 0$, non-degeneracy of $d_c$ implies $L=L'$.
\end{proof}

\section*{Prerequisites}
Recall that $\psi_t^*\lambda=e^t\lambda$ for all $t\in \R$. Consider $(\R^2(x,y),d\lambda_{\R^2})$, where $\lambda_{\R^2}=xdy$, and denote by $\{\psi^{\R^2}_t\}_{t\in \R}$ the corresponding Liouville flow. Note that the Liouville flow $\{\tilde{\psi}_t\}_{t\in \R}$ on the product $(\tilde{M},d\tilde{\lambda}):=(\R^2 \times M,d(\lambda_{\R^2}\oplus \lambda))$ is equal to the product flow $\{\psi^{\R^2}_t \times \psi_t\}_{t\in \R}$. We will use the notion of Lagrangian cobordism from \cite{BiranCornea13}. In short, given two closed Lagrangian submanifolds $L,L'\subset (M,d\lambda)$, a Lagrangian cobordism $V\subset (\tilde{M},d\tilde{\lambda})$ with negative end $L$ and positive end $L'$ is a Lagrangian submanifold $V\subset (\tilde{M},d\tilde{\lambda})$ which has two cylindrical ends of the type $((-\infty,-R]\times \{0\})\times L$ and $([R,\infty)\times \{0\})\times L'$ for some $R>0$, and satisfies the condition that $([-R,R]\times \R \times M)\cap V$ is compact (we refer to \cite{BiranCornea13} and \cite{CorneaShelukhin15} for further details on cobordisms). We write $V: L'\rightsquigarrow L$ and say that $V$ is exact if it is an exact Lagrangian submanifold of $(\tilde{M},d\tilde{\lambda})$. Given a Lagrangian cobordism $V: L'\rightsquigarrow L$  one considers its \emph{outline} $ou(V)\subset \R^2$, which by definition is the complement of the union of all unbounded connected components of $\R^2 \backslash \pi(V)$. Here $\pi: \R^2 \times M \to \R^2$ denotes the projection. The shadow $\mathcal{S}(V)$ of a Lagrangian cobordism $V: L'\rightsquigarrow L$ is then by definition the area of $ou(V)$. Given a fixed closed exact Lagrangian submanifold $L\subset (M,d\lambda)$ Cornea and Shelukhin \cite{CorneaShelukhin15} showed that 
\[
d_c(L',L''):=\inf\{\mathcal{S}(V)\ |\ V: L''\rightsquigarrow L'\ \text{is an exact Lagrangian cobordism} \}
\]
defines a non-degenerate metric on the space $\mathcal{L}(L)$. Note that "being exact Lagrangian cobordant" is an equivalence relation. To fix conventions we remark that, given $H\in C_c^{\infty}([0,1]\times M)$, we denote by $\{\phi_H^t\}_{t\in[0,1]}$ the Hamiltonian isotopy generated by the vector field $X_{H_t}$ satisfying $i_{X_{H_t}}d\lambda =-dH_t$.

\section*{Proof of Theorem \ref{corexpand1}}
The proof of Theorem \ref{corexpand1} is inspired by a technique originally employed by Biran and Cieliebak in \cite{BiranCieliebak02} for a different purpose. The first step is the following beautiful observation which (to our knowledge) was first noted by Gromov \cite{Gromov85}. We will need the proof later, so we include it here.
\begin{lemma}[\cite{Gromov85}]
	\label{propexpand10}
	If $L\subset (M,d\lambda)$ is a closed exact Lagrangian then $\{\psi_t(L)\}_{t\in \R}\subset \mathcal{L}(L)$.
\end{lemma}
\begin{proof}
	Let $f\in C^{\infty}(L)$ satisfy $\lambda|_L=df$. Consider the Lagrangian isotopy $F:\R \times L\to M$ given by $F(t,x)=\psi_t(x)$. One checks that $F^*\lambda|_{(t,x)}=e^tdf|_x$, so $F$ is an exact Lagrangian isotopy. By the proof of Theorem 3.6.7 in \cite{Oh15} or Exercise 6.1.A in \cite{Polterovich01} we conclude the existence of $H\in C^{\infty}_c([0,1]\times M)$ such that $\psi_t(L)=\phi_H^t(L)$ for all $t\in [0,1]$ and 
	\begin{equation}
	\label{eq2}
	H_t(\psi_t(x))=e^tf(x) \quad \forall \ (t,x)\in [0,1]\times L.
	\end{equation}
	The claim now follows from the Lagrangian suspension construction \cite{Oh15}, \cite{Polterovich01}.
\end{proof}
\begin{lemma}
	\label{lemexpand1}
	Fix $t\in \R$ and consider a Lagrangian cobordism $V: L'\rightsquigarrow L$. Then $\tilde{\psi}_t(V):\psi_t(L')\rightsquigarrow \psi_t(L)$ is a Lagrangian cobordism satisfying
	\begin{equation}
	\label{expand1}
	\mathcal{S}(\tilde{\psi}_t(V))=e^t\mathcal{S}(V).
	\end{equation}
\end{lemma}
\begin{proof}
The first claim is immediate from the definitions.  Using the fact that $(\psi^{\R^2}_t)^*\lambda_{\R^2}=e^t\lambda_{\R^2}$ (\ref{expand1}) follows from an easy application of the substitution formula in measure and integration theory. 
\end{proof}
Now fix two closed Lagrangian submanifolds $L,L'\subset (M,d\lambda)$. Lemma \ref{lemexpand1} shows that $\tilde{\psi}_t$ sets up a bijection between the space of Lagrangian cobordisms $L'\rightsquigarrow L$ and the space of Lagrangian cobordisms $\psi_t(L')\rightsquigarrow \psi_t(L)$. Moreover, this bijection increases the shadow (if $t>0$) or decreases the shadow (if $t<0$) in the sense of (\ref{expand1}). If in addition $L$ and $L'$ are exact then so are $\psi_t(L)$ and $\psi_t(L')$ and the bijection restricts to a bijection of exact Lagrangian cobordisms. This proves
\begin{prop}
	\label{propexpand1}
	Fix $t\in \R$ and a closed exact Lagrangian submanifold $L\subset (M,d\lambda)$. For $L'\in \mathcal{L}(L)$  we have $\psi_t(L')\in \mathcal{L}(\psi_t(L))$ and 
	\[
	d_c(\psi_t(L),\psi_t(L'))=e^td_c(L,L').
	\]
\end{prop}
\begin{proof}[Proof of Theorem \ref{corexpand1}]
	Let $L'\in \mathcal{L}(L)$. Then, by Proposition \ref{propexpand10} and \ref{propexpand1}, $\psi_t(L')\in \mathcal{L}(\psi_t(L))=\mathcal{L}(L)$. This proves that $\Psi$ is well-defined. Now apply Proposition \ref{propexpand1} again to obtain (\ref{eq1}). To prove that $\Psi$ is continuous as a map $\R \times \mathcal{L}(L)\to \mathcal{L}(L)$ we first show that, for a fixed $L' \in \mathcal{L}(L)$, the map $t\mapsto \Psi_t(L')$ is continuous. Without loss of generality we might assume that $L'=L$ and prove continuity at $t\in [0,1]$. To do this, fix $f$ and $H$ as in the proof of Proposition \ref{propexpand10} as well as $0\leq s_1 <s_2 \leq 1$. Using (\ref{eq2}) we compute that the shadow of the Lagrangian suspension corresponding to $\{\phi_H^t(L)\}_{t\in [s_1,s_2]}$ equals
	\[
	\int_{s_1}^{s_2}\max_{\psi_t(L)}(H_t)-\min_{\psi_t(L)}(H_t) dt =(\max_{L}(f)-\min_{L}(f))\int_{s_1}^{s_2}e^t dt=(\max_{L}(f)-\min_{L}(f))(e^{s_2}-e^{s_1}).
	\]  
	In particular continuity of $t\mapsto \Psi_t(L)$ follows:
	\[
	d_c(\Psi_{s_2}(L),\Psi_{s_1}(L))\leq (\max_{L}(f)-\min_{L}(f))(e^{s_2}-e^{s_1})\to 0 \quad \text{for} \ |s_2-s_1|\to 0.
	\]
	To finish the proof, fix $(s,L'')\in \R \times \mathcal{L}(L)$. Then
	\begin{align*}
		d_c(\Psi_t(L'),\Psi_s(L''))&\leq d_c(\Psi_t(L'),\Psi_t(L'')) +d_c(\Psi_t(L''),\Psi_s(L''))\\
		&=e^td_c(L',L'')+d_c(\Psi_t(L''),\Psi_s(L'')) \to 0 \quad \text{for}\ (t,L')\to (s,L'').
	\end{align*} 
\end{proof}
\section*{Relation to Hofer geometry}
Using ideas from the proof of Lemma 3.2 in \cite{BiranCieliebak02} one can also obtain versions of Theorem \ref{corexpand1} for both Lagrangian and Hamiltonian Hofer geometry. This might already be known to some extend, but we have not been able to find the explicit statement in the literature. So, since it shows a close connection between Hofer geometry and $d_c$, we sketch the argument. For details and specifics on the Hofer metric $d_h$ on $\Ham(M,d\lambda)$ we refer to \cite{Polterovich01}, whose conventions we follow here. 

Fix $T\in \R$. A small adaption of the proof of Lemma 3.2 in \cite{BiranCieliebak02} shows that, for a given Hamiltonian $H\in C_c^{\infty}([0,1]\times M)$, the Hamiltonian $e^T\psi_{-T}^*H_t(x):=e^TH_t(\psi_{-T}(x))$ generates the flow $\{\psi_T \circ \phi_H^t \circ \psi_{-T}\}_{t\in [0,1]}$. Thus, given a fixed $\phi \in \Ham(M,d\lambda)$ we have a bijection 
\begin{align*}
\{H\in C_c^{\infty}([0,1]\times M)\ |\ \phi_H^1=\phi \} &\stackrel{\approx}{\longrightarrow} \{G\in C_c^{\infty}([0,1]\times M)\ |\ \phi_G^1=\psi_T \circ \phi \circ \psi_{-T}\}\\
H&\longmapsto e^T\psi_{-T}^*H,
\end{align*}
which evidently in- or decreases oscillation by a factor of $e^T$. In particular  
\begin{equation}
\label{eq100}
d_h(\psi_T\circ \phi \circ \psi_{-T}, \id)=e^Td_h(\phi ,\id ).
\end{equation}
Defining $\Psi:\R \times \Ham(M,d\lambda)\to \Ham(M,d\lambda)$ by $(t,\phi)\mapsto \Psi_t(\phi):=\psi_t \circ \phi \circ \psi_{-t}$, equation (\ref{eq100}) reads
\[
d_h(\Psi_t(\phi),\id )=e^td_h(\phi, \id) \quad \forall \ (t,\phi)\in \R \times \Ham(M,d\lambda).
\]
Using the fact that each $\Psi_t:\Ham(M,d\lambda)\to \Ham(M,d\lambda)$ is a group isomorphism and the fact that $d_h$ is biinvariant the above equation implies 
\begin{equation}
	\label{eq101}
	d_h(\Psi_t(\phi_1),\Psi_t(\phi_2) )=e^td_h(\phi_1, \phi_2 ) \quad \forall \ t\in \R, \ \phi_1,\phi_2 \in \Ham(M,d\lambda).
\end{equation}
To check that $\Psi$ is continuous as a map $\R \times \Ham(M,d\lambda) \to \Ham(M,d\lambda)$ one can apply the same argument as in the proof of Theorem \ref{corexpand1} once it is established that, for a fixed $\phi \in \Ham(M,d\lambda)$, the map $\R \ni t\mapsto \Psi_t(\phi) \in \Ham(M,d\lambda)$ is continuous. But this is easily seen to be the case. More precisely, an easy exercise shows the existence of a Hamiltonian $H\in C^{\infty}([0,1]\times M)$ such that $\phi_H^t \circ \phi =\psi_t \circ \phi \circ \psi_{-t}$ for every $t\in [0,1]$. The existence of such $H$ can also be seen as a special case of a famous result due to Banyaga \cite[Proposition II.3.3]{Banyaga78}. We sum up these observations in the following lemma.\footnote{One does not need $(M,d\lambda)$ to be tame at infinity in order for the Hofer metric to be non-degenerate on $\Ham(M,d\lambda)$. Therefore Lemma \ref{LemmaHam} also holds without this assumption on $(M,d\lambda)$.}
\begin{lemma}
	\label{LemmaHam}
	$\{\psi_t\}_{t\in \R}$ induces a flow $\{\Psi_t\}_{t\in \R}$ on $(\Ham(M,d\lambda),d_h)$ via $\Psi_t(\phi):=\psi_t \circ \phi \circ \psi_{-t}$. More precisely $\Psi$ defines a continuous $\R$-action on $\Ham(M,d\lambda)$ and, for each $t\in \R$, $\Psi_t:\Ham(M,d\lambda)\to \Ham(M,d\lambda)$ is a group isomorphism satisfying
	\[
	d_h(\Psi_t(\phi_1),\Psi_t(\phi_2))=e^td_h(\phi_1,\phi_2) \quad \forall \ \phi_1,\phi_2 \in \Ham(M,d\lambda).
	\] 
	In particular $\id$ is the unique element of $\Ham(M,d\lambda)$ which commutes with $\psi_{t\neq 0}$ and the metric space $(\Ham(M,d\lambda),d_h)$ has infinite diameter. 
\end{lemma}

The Lagrangian Hofer geometric version of Theorem \ref{corexpand1} follows easily from arguments similar to the ones above. To state the corresponding result, recall that if $L\subset (M,d\lambda)$ is a closed Lagrangian submanifold then a result due to Chekanov \cite{Chekanov00} shows that the Lagrangian Hofer metric $d_l$ is non-degenerate on $\mathcal{H}(L):=\{\phi^1_H(L)\ |\ H\in C^{\infty}_c([0,1]\times M)\}$.
\begin{lemma}
	Let $L\subset (M,d\lambda)$ be a closed exact Lagrangian submanifold. Then $\{\psi_t\}_{t\in \R}$ induces a flow $\{\Psi_t\}_{t\in \R}$ on $(\mathcal{H}(L),d_l)$ via $\Psi_t(L'):=\psi_t(L')$. More precisely, $\Psi$ defines a continuous $\R$-action on $\mathcal{H}(L)$ and every $\Psi_t:\mathcal{H}(L)\to \mathcal{H}(L)$ satisfies
	\begin{equation*}
	d_l(\Psi_t(L'),\Psi_t(L''))=e^td_l(L',L'') \quad \forall \ L',L''\in \mathcal{H}(L).
	\end{equation*}
	In particular the metric space $(\mathcal{H}(L),d_l)$ has infinite diameter.
\end{lemma}
\subsection*{Acknowledgement} I am very grateful to my advisor Paul Biran for suggesting I consider how one can apply the techniques from the proof of Lemma 3.2 in \cite{BiranCieliebak02} to a Lagrangian cobordism, and for pointing out Gromov's observation to me. This advice triggered the idea presented above. I am also very grateful to Egor Shelukhin for his interest in this work and for several illuminating discussions.
\bibliographystyle{plain}
\bibliography{BIB}
\end{document}